\documentclass[a4,12pt]{amsart}
\usepackage{amssymb,amstext,amsmath,amscd,amsthm,amsfonts,enumerate,graphicx,latexsym}
\usepackage[usenames]{color}
\usepackage[all]{xy}

 \usepackage{comment}
\newtheorem{thm}{Theorem}[section]
\newtheorem{cor}[thm]{Corollary}
\newtheorem{lem}[thm]{Lemma}
\newtheorem{prop}[thm]{Proposition}
\theoremstyle{definition}
\newtheorem{dfn}[thm]{Definition}
\newtheorem{rem}[thm]{Remark}
\newtheorem{ques}[thm]{Question}
\newtheorem{conj}[thm]{Conjecture}
\newtheorem{ex}[thm]{Example}

\newtheorem{claim}{Claim}
\newtheorem*{claim*}{Claim}
\theoremstyle{remark}
\newtheorem*{ac}{Acknowlegments}

\numberwithin{equation}{thm}
\def\add{\operatorname{add}}
\def\ann{\operatorname{ann}}

\def\assh{\operatorname{Assh}}
\def\cm{\operatorname{CM}}
\def\codim{\operatorname{codim}}
\def\cok{\operatorname{Coker}}

\def\End{\operatorname{End}}
\def\Ext{\operatorname{Ext}}
\def\height{\operatorname{ht}}
\def\h{\operatorname{H}}
\def\Hom{\operatorname{Hom}}

\def\id{\mathrm{id}}
\def\im{\operatorname{Im}}
\def\ker{\operatorname{Ker}}
\def\lend{\operatorname{\underline{End}}}
\def\lhom{\operatorname{\underline{Hom}}}
\def\lmod{\operatorname{\underline{mod}}}
\def\m{\mathfrak{m}}

\def\mod{\operatorname{mod}}
\def\n{\mathfrak{n}}
\def\ng{\operatorname{NG}}
\def\p{\mathfrak{p}}
\def\res{\operatorname{res}}
\def\soc{\operatorname{soc}}
\def\spec{\operatorname{Spec}}
\def\syz{\Omega}

\def\Tor{\operatorname{Tor}}
\def\Tr{\operatorname{Tr}}
\def\tr{\operatorname{tr}}
\def\uend{\operatorname{\overline{End}}}
\def\uhom{\operatorname{\overline{Hom}}}
\def\V{\operatorname{V}}
\begin{document}
\setlength{\baselineskip}{15pt}
\title[traces of canonical modules and annihilator of Ext]{Trace of canonical modules, annihilator of Ext, and classes of rings close to being Gorenstein}
\author{Hailong Dao}
\address{Department of Mathematics, University of Kansas, Lawrence, KS 66045-7523, USA}
\email{hdao@ku.edu}
\urladdr{https://www.math.ku.edu/~hdao/}
\author{Toshinori Kobayashi}
\address{Graduate School of Mathematics, Nagoya University, Furocho, Chikusaku, Nagoya, Aichi 464-8602, Japan}
\email{m16021z@math.nagoya-u.ac.jp}
\author{Ryo Takahashi}
\address{Graduate School of Mathematics, Nagoya University, Furocho, Chikusaku, Nagoya, Aichi 464-8602, Japan}
\email{takahashi@math.nagoya-u.ac.jp}
\urladdr{http://www.math.nagoya-u.ac.jp/~takahashi/}
\thanks{Hailong Dao was partly supported by Simons Collaboration Grant FND0077558.
Toshinori Kobayashi was partly supported by JSPS Grant-in-Aid for JSPS Fellows 18J20660.
Ryo Takahashi was partly supported by JSPS Grant-in-Aid for Scientific
Research 19K03443.}
\begin{abstract}
In this note we study trace ideals of canonical modules.
Characterizations of the trace ideals in terms of annihilators of certain Ext modules are given.
We apply our results to study many classes of rings close to being Gorenstein that appear in recent literature.
We discuss the behavior of nearly Gorensteinness, introduced by Herzog-Hibi-Stamate, under reductions by regular sequences.
A nearly Gorenstein version of the Tachikawa conjecture is posed, and an affirmative answer given for type $2$ rings.
We also introduce the class of weakly almost Gorenstein rings using the canonical module, and show that they  are  almost Gorenstein rings in dimension one, and are closely related to Teter rings in the artinian case.
We explain their basic properties and give some examples.
\end{abstract}
\maketitle
\section{Introduction}

Let $R$ be a commutative ring.
Recall that for an $M$-module, its trace ideal $\tr M$ is a sum of images of all homomorphisms from $M$ to $R$.
Trace ideals have been studied in various situations and have enjoyed a surge in popularity in recent years, see \cite{GIK,lp,L} and the references therein.

The trace ideals of canonical modules over Cohen--Macaulay rings are quite important tools to understand the difference between Gorenstein rings and non-Gorenstein rings.
For example, it is known that the trace ideal of the canonical module defines a non-Gorenstein locus of the ring; see \cite[Lemma 1.4]{D}.
There are many works on the trace ideals of canonical modules \cite{HHS,HMP,HV}.

In this paper, we  compare trace ideals of canonical modules with  annihilators of certain Ext modules.
Our first main result is the following.

\begin{thm} \label{thA}
Let $(R,\m)$ be a $d$-dimensional Cohen--Macaulay local ring with a canonical module $\omega$.
Then equalities $\tr \omega = \ann \Ext^{>0}_R(\omega,\mod R) = \ann \Ext^1_R(\cm(R),R)$ hold.
Additionally assume that $R$ is Gorenstein on the punctured spectrum.
Then an equality $\tr \omega=\ann \Ext^{d+1}_R(\Tr \omega,R)$ holds, where $\Tr \omega$ is the transpose of $\omega$.
\end{thm}

For the explanation of the notations and the proof, see Definition \ref{d1}, Theorem \ref{m} and Proposition \ref{1}.

Using Theorem \ref{thA}, we study nearly Gorenstein rings.
Recall that a Cohen--Macaulay local ring $(R,\m)$ with a canonical module $\omega$ is \textit{nearly Gorenstein} if $\tr \omega$ contains the maximal ideal $\m$ \cite{HHS}.
We consider a situation that $\underline{x}$ is a regular sequence of a local ring $R$ and $R/(\underline{x})$ is nearly Gorenstein.
In general, $R$ may not be nearly Gorenstein.
However, if we take suitable regular sequences $\underline{x}$, we can characterize the nearly Gorensteinness of $R$ by that of $R/(\underline{x})$.
Indeed, we prove the following theorem.

\begin{thm} \label{thB}
Let $(R,\m)$ be a Cohen--Macaulay local ring with a canonical module $\omega$.
Then the following conditions are equivalent.
\begin{enumerate}[\rm(1)]
\item $R$ is nearly Gorenstein.
\item $R/(\underline{x})$ is nearly Gorenstein for any $R$-regular sequence $\underline{x}=x_1,\dots,x_t$.
\item $R$ admits an $R$-regular sequence $\underline{x}=x_1,\dots,x_t$ in $\m^2$ such that $R/(\underline{x})$ is nearly Gorenstein.
\item $R$ admits an $R$-regular sequence $\underline{x}=x_1,\dots,x_t$ in $\tr\omega$ such that $R/(\underline{x})$ is nearly Gorenstein.
\end{enumerate}
\end{thm}

Note that the implication \textrm{(1)} $\Rightarrow$ \textrm{(2)} in Theorem \ref{thB} is already shown in \cite[Proposition 2.3]{HHS}.
The almost Gorenstein local rings of dimension one is a notable subclass of nearly Gorenstein rings.
While trying to give similar characterizations of almost Gorenstein local rings as in Theorem \ref{thB}, we discover the class of weakly almost Gonrestein rings.
We say that a Cohen--Macaulay local ring $(R,\m)$ of dimension at most $1$  is \textit{weakly almost Gonrestein ring} if it has a canonical module $\omega$ and an exact sequence $R\to \omega \to (R/\m)^{\oplus n}\to 0$ for some $n$.
Note that the first map might not be injective.
We see that the class of almost Gorenstein local rings of dimension one fits into the following characterization.

\begin{thm} \label{thC}
Let $R$ be a Cohen--Macaulay local ring having a canonical module $\omega$ and Krull dimension $1$.
Then the followings are equivalent.
\begin{enumerate}[\rm(1)]
\item $R$ is an almost Gorenstein ring.
\item $R$ is weakly almost Gorenstein.
\item $R/(x)$ is weakly almost Gorenstein for any $R$-regular element $x\in\m$.
\item $R/(x)$ is weakly almost Gorenstein for some $R$-regular element $x\in\m^2$.
\end{enumerate}
\end{thm}

Note that by the above theorem, we can easily construct a weakly almost Gorenstein ring $R/(x)$ from an almost Gorenstein ring of dimension one and a regular element $x$.
We also provide some examples of weakly almost Gorenstein artinian rings of embedding dimension three (Example \ref{e51} and \ref{e52}).

On the other hand, weakly almost Gorenstein artinian rings are closely related to another well-studied class of rings called Teter rings, studied by Teter, Huneke-Vraciu, and others (\cite{An, HV}). 
\begin{thm} \label{thD}
Let $(R,\m,k)$ be an artinian local ring such that $k$ is infinite. 
The following conditions are equivalent.
\begin{enumerate}[\rm(1)]
\item $R$ is weakly almost Gorenstein.
\item There exists an ideal $I\subseteq \soc R$ such that $R/I$ is Gorenstein.
\item $R/\soc R$ is Teter.
\item $R/\soc R$ has self-canonical dual maximal ideal (i.e. $\Hom_R(\m/\soc R,\omega)\cong \m/\soc R$.)
\end{enumerate}
\end{thm}
The organization of this paper is as follows.
In section 2, we collect preliminary notions, and prove Theorem \ref{thA} (Theorem \ref{m} and Proposition \ref{1}).
In section 3, we discuss the behavior of trace ideals of canonical modules under reductions by regular sequences.
Especially, we give the proof of Theorem \ref{thB}.
In section 4, we pose a question about nearly Gorenstein rings in the context of Tachikawa conjecture.
We give a partial answer for the question in the case of Cohen--Macaulay type $2$.
In section 5, we study weakly almost Gorenstein rings.
Theorem \ref{thC} and \ref{thD} are proved in this section.

\section{Trace ideal and annihilation of Ext modules}
Throughout this paper, we assume that all rings are commutative and noetherian, and that all modules are finitely generated.
We allow an empty set to be a regular sequence of an module.
For a ring $R$, we denote by $\mod R$ the category of $R$-modules.
If $R$ is Cohen--Macaulay and local, we denote by $\cm R$ the category of maximal Cohen--Macaulay $R$-modules.

In this section we give some characterizations of trace ideals of canonical modules.

\begin{dfn} \label{d1}
Let $R$ be a ring.
\begin{enumerate}[(1)]
\item
Set $(-)^*=\Hom_R(-,R)$.
The {\em trace ideal} $\tr M$ of an $R$-module $M$ is defined by:
$$
\tr M=\sum_{f\in M^*}\im f=(f(x)\mid f\in M^*,\,x\in M).
$$
\item
For $R$-modules $M,N$ we denote by $\lhom_R(M,N)$ the quotient of $\Hom_R(M,N)$ by the maps $M\to N$ factoring through a free module.
We set $\lend_R(M)=\lhom_R(M,M)$.
\item We denote by $\lmod R:=\mod R/[\add R]$ the stable category of $R$ \cite{ABr}.
Here $\add R$ is a full subcategory of $\mod R$ consisting of all projective $R$-modules.
\item
Let $R$ be a Cohen--Macaulay ring with canonical module $\omega$.
For $R$-modules $M,N$ we denote by $\uhom_R(M,N)$ the quotient of $\Hom_R(M,N)$ by the maps $M\to N$ factoring through a direct sum of copies of $\omega$.
We set $\uend_R(M)=\uhom_R(M,M)$.
\item 
For an $R$-module $M$ and integer $i\ge 0$, $\ann \Ext^i_R(M,\mod R)$ consists of elements $a\in R$ which annihilates $\Ext^i_R(M,N)$ for all $N\in\mod R$.
We use similar notations for $\Tor$, $\lhom$ and $\uhom$.
Also, $\ann \Ext^1_R(\cm(R),M)$ consists of elements $a\in R$ such that $a\Ext^1(N,M)=0$ for all $N\in\cm(R)$.
Moreover, we define $\ann \Ext^{>0}_R(M,\mod R)$ to be the intersection $\bigcap_{i>0} \ann \Ext^i_R(M,\mod R)$.
\item For an $R$-module $M$, let $P_1\xrightarrow[]{\varphi} P_0 \to M \to 0$ be a exact sequence with projective modules $P_1,P_0$.
The \textit{first syzygy} module $\syz M$ of $M$ is defined as the image of $\varphi$.
Also we define \textit{$n$-th syzygy} module $\syz^n M$ inductively as $\syz(\syz^{n-1}M)$.
The \textit{(Auslander) transpose} $\Tr M$ of $M$ is defined as the cokernel of the homomorphism
$\varphi^*\colon\Hom_R(P_0,R)\to\Hom_R(P_1,R)$.
Note that the isomorphism classes of $\syz^n M$ and $\Tr M$ up to projective direct summands are independent of the choise of $\varphi$.
Moreover $\syz$ gives an endofunctor $\lmod R\to \lmod R$ and $\Tr$ gives a duality of $\lmod R$ \cite{ABr}.
\end{enumerate}
\end{dfn}

The following lemma is shown by \cite[Lemma 2.14]{ua} and its proof.

\begin{lem}\label{2}
Let $R$ be a ring and $M$ an $R$-module.
Then there are equalities $\ann\lend_R(M)=\ann\Ext^1_R(M,\syz M)=\ann\Ext^{>0}_R(M,\mod R)$.
\end{lem}

The following theorem gives characterizations of trace ideals of canonical modules.
It includes the first equalities of Theorem \ref{thA}. 

\begin{thm}\label{m}
Let $(R,\m)$ be a $d$-dimensional Cohen--Macaulay local ring with canonical module $\omega$.
The following sets of elements of $R$ are all equal.
\begin{enumerate}[\rm(1)]
\item
$\tr\omega$,
\item
The set of elements $a\in R$ such that the map $R\xrightarrow{a}R$ factors through a direct sum of copies of $\omega$,
\item
The set of elements $a\in R$ such that the map $\omega\xrightarrow{a}\omega$ factors through a free module,
\item
$\ann\lend_R(\omega)=\ann\Tor_1^R(\Tr\omega,\omega)$,
\item
$\ann\lhom_R(\omega,\mod R)=\ann\Tor_1^R(\Tr\omega,\mod R)$,
\item
$\ann\lhom_R(\mod R,\omega)=\ann\Tor_1^R(\mod R,\omega)$,
\item
$\ann\Ext^1_R(\omega,\syz\omega)$,
\qquad {\rm (8)} $\ann\Ext^{>0}_R(\omega,\mod R)$,
\qquad {\rm (9)} $\ann\Ext^1_R(\cm(R),R)$,
\item[{\rm (10)}]
$\ann\Ext^{d+1}_R(\mod R,R)$,
\qquad {\rm (11)} $\ann\uend_R(R)$,
\qquad {\rm (12)} $\ann\uhom_R(R,\mod R)$,
\item[{\rm (13)}]
$\ann\uhom_R(\mod R,R)$.
\end{enumerate}
\end{thm}

\begin{proof}
The equalities in (4), (5) and (6) follow from \cite[Lemma 3.9]{Y} and the fact that $\Tr$ gives a duality of $\lmod R$.
We set $(-)^\dag=\Hom_R(-,\omega)$.

$(1)=(2)$:
See \cite[Lemma 11.41]{LW} for instance.

$(2)=(3)$:
Application of the functor $(-)^\dag$ shows this; note that $\Hom_R(\omega,\omega)\cong R$.

$(3)=(4)=(5)=(6)$:
Note that for $R$-modules $M,N$ we have $\lhom_R(M,N)=0$ if either $M$ or $N$ is free.
Let $a\in R$.
If the map $\omega\xrightarrow{a}\omega$ factors through a free module, $\lhom_R(M,\omega)\xrightarrow{a}\lhom_R(M,\omega)$ and $\lhom_R(\omega,M)\xrightarrow{a}\lhom_R(\omega,M)$ are zero maps for all $R$-modules $M$.
On the other hand, if $a$ annihilates $\lend_R(\omega)$, then $a\cdot\id_{\omega}=(\omega\xrightarrow{a}\omega)$ factors through a free module.

$(2)=(11)=(12)=(13)$:
These equalities follow from the dual argument to the proof of the equalities $(3)=(4)=(5)=(6)$.

$(4)=(7)=(8)$:
Lemma \ref{2} show this.

$(8)\subseteq(9)\subseteq(7)$:
We have
\begin{align*}
&\ann\Ext^{>0}_R(\omega,\mod R)\subseteq\ann\Ext^1_R(\omega,\cm(R))=\ann\Ext^1_R(\cm(R),R),\\
&\ann\Ext^1_R(\omega,\cm(R))\subseteq\ann\Ext^1_R(\omega,\syz\omega).
\end{align*}
Here the equality follows from the isomorphism
\begin{equation}\label{e}
\Ext^1_R(M,R)\cong\Ext^1_R(\omega,M^\dag)
\end{equation}
for each $M\in\cm(R)$, which is shown by \cite[(A.4.22)]{C}, and the fact that $(-)^\dag$ gives a duality of $\cm(R)$.
The second inclusion holds since $\syz\omega$ is in $\cm(R)$.

$(9)\subseteq(10)$:
There is an isomorphism $\Ext^{d+1}_R(M,R)\cong\Ext^1_R(\syz^dM,R)$, and $\syz^dM$ belongs to $\cm(R)$.

$(10)\subseteq(9)$:
Pick an element $a\in\ann\Ext^{d+1}_R(\mod R,R)$.
What we need to show is that $a\Ext^1_R(M,R)=0$ for all $M\in\cm(R)$.
Let us do this by induction on $d$.
When $d=0$, there is nothing to prove.
Assume $d>0$, and let $x\in\m$ be any $R$-regular element.
Then the exact sequence
$$
0 \to M \xrightarrow{x} M\to M/xM \to 0
$$
induces a commutative diagram
$$
\begin{CD}
\Ext_R^1(M/xM,R) @>>> \Ext_R^1(M,R) @>x>> \Ext_R^1(M,R) @>\delta>> \Ext_R^2(M/xM,R)\\
@. @VaVV @VaVV @VaVV\\
@. \Ext_R^1(M,R) @>x>> \Ext_R^1(M,R) @>\delta>> \Ext_R^2(M/xM,R)
\end{CD}
$$
with exact rows.
Since $a\Ext_R^{d+1}(\mod R,R)=0$, we have $a\Ext_R^{d+1}(\mod R/(x),R)=0$, and hence $a\Ext_{R/(x)}^d(\mod R/(x),R/(x))=0$ by \cite[Lemma 3.1.16]{BH}.
The induction hypothesis implies that the multiplication map
$$
\Ext_{R/(x)}^1(M/xM,R/(x))\xrightarrow{a}\Ext_{R/(x)}^1(M/xM,R/(x))
$$
is zero.
This map can be identified with the multiplication map
$$
\Ext_R^2(M/xM,R)\xrightarrow{a}\Ext_R^2(M/xM,R),
$$
and therefore the right-end vertical map in the above diagram is zero.
This implies that for each element $z\in\Ext_R^1(M,R)$ we have $\delta(az)=0$, and $az\in\ker\delta=x\Ext_R^1(M,R)$.
Replacing $x$ with its any power $x^i$, we have $az\in x^i\Ext_R^1(M,R)$ for all $i>0$.
Thus $az\in\bigcap_{i>0}x^i\Ext_R^1(M,R)=0$ by Krull's intersection theorem.
It follows that $a\Ext_R^1(M,R)=0$, which is what we want.
\end{proof}

As a corollary, we get the following inclusion of ideals.
We use this result in Section 4.

\begin{cor} \label{c14}
Let $R$ be a Cohen--Macaulay local ring with canonical module $\omega$.
Then $\tr\omega$ is contained in $\ann \Ext^i_R(\omega,R)$ for all $i>0$.
\end{cor}

\begin{proof}
The assertion follows by Theorem \ref{m} (1) $=$ (8).
\end{proof}

The following notions are introduced in \cite{HHS} and \cite{SV}.

\begin{dfn}[Herzog--Hibi--Stamate]
Let $(R,\m)$ be a Cohen--Macaulay local ring with canonical module $\omega$.
Then $R$ is called {\em nearly Gorenstein} if $\tr\omega$ contains $\m$.
\end{dfn}

\begin{dfn}[Striuli--Vraciu]
Let $(R,\m)$ be an artinian local ring.
Then $R$ is called {\em almost Gorenstein} if $0:(0:I)\subseteq I:\m$ for all ideals $I$ of $R$.
\end{dfn}

Note that, for an artinian ring $R$, it is shown in \cite[Proposition 1.1]{HV} that $R$ is almost Gorenstein if it is nearly Gorenstein.

Theorem \ref{m} immediately yields the following.

\begin{cor}\label{5}
Let $(R,\m)$ be a Cohen--Macaulay local ring with canonical module $\omega$.
Then $R$ is nearly Gorenstein if and only if $\m\Ext^1(\cm(R),R)=0$.
\end{cor}

The following result is given in \cite[Proposition 2.2]{SV}, which follows from combining our Corollary \ref{5} with \cite[Proposition 2.3(b)]{HHS}.

\begin{cor}[Striuli--Vraciu]
Let $(R,\m)$ be a Cohen--Macaulay local ring with canonical module.
If $\m\Ext^1(\cm(R),R)=0$, then $R/Q$ is almost Gorenstein for all parameter ideals $Q$ of $R$.
\end{cor}

\begin{rem}
In \cite[Proposition 2.2]{SV} the existence of a canonical module is not assumed.
\end{rem}

Now we aim to show the second equality of Theorem \ref{thA} with an assumption that the ring is Gorenstein on the punctured spectrum.
We prepare the following lemma.

\begin{lem}\label{4}
Let $(R,\m)$ be a $d$-dimensional Cohen--Macaulay local ring with canonical module $\omega_R$.
Let $M$ be an $R$-module of finite length.
Then $\ann_RM=\ann_R\Ext_R^d(M,\omega_R)$.
\end{lem}

\begin{proof}
Clearly, $\ann_RM$ is contained in $\ann_R\Ext_R^d(M,\omega_R)$.
Put $N=\Ext_R^d(M,\omega_R)$ and take an element $a\in\ann_RN$.
There are isomorphisms
$$
\widehat N\cong\Ext_{\widehat R}^d(\widehat M,\omega_{\widehat R})\cong(\h_{\widehat\m}^0(\widehat M))^\vee=(\widehat M)^\vee,
$$
where $\widehat{()}$ denotes the $\m$-adic completion, $(-)^\vee$ the Matlis dual, the second isomorphism follows from the local duality theorem, and the equality comes from the fact that $\widehat M$ has finite length.
We have $a\widehat N=0$, hence $a(\widehat M)^\vee=0$, and therefore $aM=0$.
\end{proof}

The following proposition completes the Theorem \ref{thA}.

\begin{prop}\label{1}
Let $(R,\m)$ be a $d$-dimensional Cohen--Macaulay local ring with canonical module $\omega$.
Suppose that $R$ is locally Gorenstein on the punctured spectrum $\spec R\setminus\{\m\}$.
Then $\tr\omega=\ann\Ext^{d+1}_R(\Tr\omega,R)$.
\end{prop}

\begin{proof}
In our setting, $\omega$ is locally free on the punctured spectrum.
In particular, $\lend_R(\omega)=\Tor^R_1(\Tr\omega,\omega)$ has finite length as an $R$-module.
There are isomorphisms
$$
\Ext^{d+1}_R(\Tr\omega,R)\cong\Ext^1_R(\syz^d\Tr\omega,R)\cong\Ext^1_R(\omega,(\syz^d\Tr\omega)^\dag)\cong\Ext^d_R(\lend_R(\omega),\omega),
$$
where $(-)^\dag=\Hom_R(-,\omega)$, and the second and third isomorphisms follow from \eqref{e} and \cite[Lemma 3.10]{Y}, respectively.
Lemma \ref{4} and Theorem \ref{m} yield equalities $\ann\Ext^{d+1}_R(\Tr\omega,R)=\ann\lend_R(\omega)=\tr\omega$.
\end{proof}

In the rest of this section, we deal with the case that $R$ is not necessarily Gorenstein on the punctured spectrum.
First we put the following lemma.

\begin{lem} \label{l112}
Let $R$ be a ring and $M$ an $R$-module.
Then an inclusion $\ann \lend_R(M) \subseteq \ann\Ext^1_R(\Tr M,R)$ holds.
\end{lem}

\begin{proof}
We set $(-)^*=\Hom_R(-,R)$.
The $R$-module $\Ext^1_R(\Tr M,R)$ is isomorphic to the kernel $K$ of the homomorphism $M \to M^{**}$ which is given by $x\mapsto [f\mapsto f(x)]$ for $x\in M$ and $f\in M^*$; See \cite[Proposition 12.5]{LW} for instance.
Note that $K$ coincides with the submodule $\bigcap_{f\in M^*} \ker f$ of $M$.
Let $x\in \ann \lend_R(M)$.
Then the multiplication map $x\colon M \to M$ factors through a free module $F$.
By the identification above, the image of $K$ in $F$ vanishes.
This yields that $xK=0$.
Therefore $\ann\lend_R(M)\subseteq \ann K =\ann \Ext^1_R(\Tr M,R)$.
\end{proof}

Recall that for an integer $n\ge 1$, an $R$-module $M$ is called \textit{$n$-torsionfree} \cite{ABr} if $\Ext^i_R(\Tr M,R)=0$ for all $i=1,\dots, n$.

\begin{lem} \label{l213}
Let $R$ be a ring and $M$ an $R$-module.
Assume $M$ is $n$-torsionfree.
Then an inclusion $\ann \lend_R(M) \subseteq \ann\Ext^{n+1}_R(\Tr M,R)$ holds.
\end{lem}

\begin{proof}
Let $\mathcal{X}_n$ be the subcategory $\{X\in \mod R\mid \Ext^i_R(X,R)=0\text{ for } i=1,\dots,n\}$ of $\mod R$, $\mathcal{F}_n$ be the subcategory $\{X\in\mod R\mid \Tr X\in \mathcal{F}_n\}$ of $\mod R$, and $\underline{\mathcal{X}}_n,\underline{\mathcal{F}}_n$ be corresponding subcategories of $\lmod R$.
It follows from \cite[Prop. 1.1.1]{I} that the functor $\Tr \syz^n \Tr$ gives an equivalence $\underline{\mathcal{F}}_n\cong \underline{\mathcal{X}}_n$.
In particular, for any $R$-module $X\in \mathcal{F}_n$, one has an isomorphism $\lend_R(X)\cong \lend_R(\Tr \syz^n \Tr X)$ of $R$-modules.

Set $N:=\Tr \syz^n \Tr M$.
Since $M$ is $n$-torsionfree, $X$ belongs to $\mathcal{F}_n$, and hence $\lend_R(N)\cong \lend_R(M)$.
The calculations
$$\Ext^{n+1}_R(\Tr M,R)\cong\Ext^1_R(\syz^n \Tr M,R)\cong\Ext^1_R(\Tr(\Tr\syz^n \Tr M),R)=\Ext^1_R(\Tr N,R)$$
and Lemma \ref{l112} show that $\ann\lend_R(N) \subseteq \ann\Ext^{n+1}_R(\Tr M,R)$.
Thus we conclude that $\ann \lend_R(M) \subseteq \ann\Ext^{n+1}_R(\Tr M,R)$.
\end{proof}

The $n$-torsionfreeness of a canonical module can be seen as follows.
This is well-known to experts.
For the convenience of the reader, we include a proof.

\begin{lem} \label{l214}
Let $R$ be a $d$-dimensional Cohen--Macaulay local ring with a canonical module $\omega$.
Then $\omega$ is $d+1$-torsionfree if and only if $R$ is Gorenstein.
Consequently, for $n\ge 1$, $\omega$ is $n$-torsionfree if and only if $R_\p$ is Gorenstein for all prime ideal $\p$ of $R$ with $\height \p\le n-1$.
\end{lem}

\begin{proof}
The ``if" part is obvious.
Conversely assume that $\omega$ is $d+1$-torsionfree.
Then there exists an $R$-module $X\in\cm R$ with a short exact sequence $0\to \omega \to F \to X \to 0$, where $F$ is a free module; see \cite[Proposition 2.4]{IW}.
Since $\Ext^1_R(X,\omega)$ vanishes, this exact sequence splits.
It follows that $\omega$ is a summand of $F$, which implies that $\omega$ is isomorphic to $R$.
Thus $R$ is Gorenstein.
\end{proof}

We denote by $\ng(R)$ the {\em non-Gorenstein locus} of $R$, that is, the set of prime ideals $\p$ of $R$ such that $R_\p$ is not Gorenstein.

\begin{cor}
Let $R$ be a Cohen--Macaulay local ring with canonical module $\omega$.
Put $h=\codim \ng(R)$.
Then an inclusion $\tr\omega \subseteq \ann\Ext^{h+1}_R(\Tr \omega,R)$ holds.
\end{cor}

\begin{proof}
Note that $\omega$ is $h$-torsionfree by Lemma \ref{l214}.
Therefore the assertion follows by Lemma \ref{l213}.
\end{proof}

We close this section by posing some questions related to the corollary above, and put some remarks on them.

\begin{ques}\label{3}
Let $R$ be a $d$-dimensional Cohen--Macaulay local ring with canonical module $\omega$.
\begin{enumerate}[(1)]
\item
Does there exist an $R$-module $M$ such that $\tr\omega=\ann\Ext^{d+1}_R(M,R)$?
\item
Put $h=\codim\ng(R)$.
Then, does the equality $\tr\omega=\ann\Ext^{h+1}_R(\Tr\omega,R)$ hold?
\end{enumerate}
\end{ques}

\begin{rem}
\begin{enumerate}[(1)]
\item
Proposition \ref{1} says that Question \ref{3}(1) is affirmative if $R$ is locally Gorenstein on the punctured spectrum.
\item
The unmixed parts of the equality in Question \ref{3}(2) coincide.
In fact, let $I$ (resp. $J$) be the unmixed part of the left (resp. right) hand side.
Take any $\p\in\assh_R(R/\tr\omega)$.
Then we have $\height\p=h$ and see that $R_\p$ is locally Gorenstein on the punctured spectrum.
Proposition \ref{1} implies that
$$
I_\p=(\tr_R\omega)_\p=\tr_{R_\p}\omega_\p=\ann_{R_\p}\Ext_{R_\p}^{h+1}(\Tr_{R_\p}\omega_\p,R_\p)=(\ann_R\Ext_R^{h+1}(\Tr_R\omega,R))_\p.
$$
\item
The equality in Proposition \ref{1} is not true in general.
Let $R=k[[x,y,z]]/(x,y)^3$ with $k$ a field.
Then $d=1$, and we have $\tr\omega=(x,y)^2$, while $\ann\Ext^{d+1}(\Tr\omega,R)=(x,y)$.
On the other hand, $R$ satisfies the equality in Question \ref{3}(2); one has $\ng(R)=\V(0)$ and $h=0$.
\end{enumerate}
\end{rem}


\section{Reduction by a regular sequence}

We recall the notion of residues for Cohen--Macaulay local rings introduced in \cite{HHS}.

\begin{dfn}[Herzog--Hibi--Stamate]
Let $R$ be a Cohen--Macaulay local ring with a canonical module $\omega$.
Assume $R$ is locally Gorenstein on the punctured spectrum so that $\tr(\omega)$ is $\m$-primary.
Then we call the colength $\ell_R(R/\tr(\omega))$ the \textit{residue} of $R$, and denote by $\res(R)$.
\end{dfn}

It directly follows from the definition that $R$ is Gorenstein if and only if $\res(R)=0$,
and $R$ is nearly Gorenstein if and only if $\res(R)\le 1$.
Some upper bound of $\res(R)$ for an artinian ring $R$ is given by Ananthnarayan \cite{An}.
On the trace ideal of $\omega$ and reductions by regular sequences, the following lemma is fundamental.

\begin{lem} \label{l22}
Let $R$ be a Cohen--Macaulay local ring with a canonical module $\omega_R$ and an $R$-regular sequence $(\underline{x})$ in $\m$.
Set $R'=R/(\underline{x})$.
Then a containment $\tr_R(\omega_R)R'\subseteq \tr_{R'}(\omega_{R'})$ holds.
In particular, one has an inequality $\res(R)\ge \res(R')$.
\end{lem}

\begin{proof}
We refer to \cite[Proposition 2.3]{HHS} and its proof.
\end{proof}

\begin{rem}
In general, $\res(R)$ might be strictly bigger than $\res(R')$.
Indeed, let $R$ be the local numerical semigroup ring $k[\![t^3,t^7,t^8]\!]$ over a field $k$.
Then $R$ is an integral domain of dimension one, and so it is locally Gorenstein on the punctured spectrum.
The matrix $A=\begin{pmatrix} x_1^2 & x_2 & x_3\\ x_2 & x_3 & x_1^3\end{pmatrix}$ provides a minimal free resolution
$$0\to S^2 \xrightarrow[]{A^t} S^3 \to S \to R \to 0$$
of $R$ over a regular local ring $S=k[\![x_1,x_2,x_3]\!]$.
Since $R$ has Cohen--Macaulay type $2$, it follows from \cite[Corollary 3.4]{HHS} that $\tr(\omega_R)=I_1(A)R=(x_1^2,x_2,x_3)R=(t^6,t^7,t^8)\not=\m$ (Note that more general computation of $\tr(\omega_R)$ for a numerical semigroup ring $R$ of embedding dimension three is given in \cite[Proposition 2.1]{HHS2}).
On the other hand, the element $t^3\in \m$ is $R$-regular, and $R/(t^3)$ is isomorphic to $k[x_2,x_3]/(x_2,x_3)^2=:R'$.
It is easy to check that $R'$ is nearly Gorenstein.
Thus we have $\res(R/(t^3))=1<\res(R)=2$.
\end{rem}


On the other hand, we claim in Theorem \ref{thB} that a local ring $(R,\m)$ is nearly Gorenstein if $\underline{x}$ is a regular sequence in $\m^2$ and $R/(\underline{x})$ is nearly Gorenstein.
We give the proof of Theorem \ref{thB} in the followings by using Theorem \ref{m}.

\begin{proof}[Proof of Theorem \ref{thB}]
In all directions, we may assume $R$ is non-Gorenstein.

(1) $\Rightarrow$ (2): This implication follows by Lemma \ref{l22}.
(2) $\Rightarrow$ (3) and (4): These implications are clear.
(3) $\Rightarrow$ (1): By an induction argument, we reduce the proof to the case when $t=1$.

Hence we may assume $\underline{x}=x$ and $R':=R/(x)$ is nearly Gorenstein.
We aim to show that $R$ is nearly Gorenstein.
Fix an $R$-module $M\in\cm(R)$.
Take an element $y\in \m$.
Consider a commutative diagram
\[
\xymatrix{
\cdots \ar[r] & \Ext^1_R(M,R) \ar[r]^-{x} \ar[d]^-y & \Ext^1_R(M,R) \ar[r] \ar[d]^-y & \Ext^2_R(M/x M,R) \ar[r] \ar[d]^-y & \cdots\\
\cdots \ar[r] & \Ext^1_R(M,R) \ar[r]^-{x} & \Ext^1_R(M,R) \ar[r] & \Ext^2_R(M/x M,R) \ar[r] & \cdots
}
\]
with exact rows induced by a short exact sequence $0\to M \to M \to M/xM \to 0$.
Note that $\Ext^2_R(M/x M,R')$ is isomorphic to $\Ext^1_{R'}(M/xM,R')$.
Since $R'$ is nearly Gorenstein, $y$ annihilates $\Ext^1_{R'}(\cm(R'),R')$ by Theorem \ref{m}.
Therefore the diagram above yields that $y\Ext^1_R(M,R)$ is contained in $x\Ext^1_R(M,R)$.
It follows that $\m\Ext^1_R(M,R)$ is contained in $x\Ext^1_R(M,R)$.
Since $x$ belongs to $\m^2$, Nakayama's lemma implies that $\m\Ext^1_R(M,R)=0$.
Thus $\m\Ext^1_R(\cm(R),R)=0$.
By Theorem \ref{m}, this equality means that $R$ is nearly Gorenstein.

(4) $\Rightarrow$ (1): We only need to prove the assertion in the case that $t=1$ and $\underline{x}=x$.
Suppose that $R':=R/(x)$ is nearly Gorenstein.
Fix an $R$-module $M\in\cm(R)$.
Since $x\in \tr\omega$, $x$ annihilates $\Ext^1_R(M,R)$ by Theorem \ref{m}.
We have an injection $\Ext^1_R(M,R) \to \Ext^1_{R'}(M/xM,R')$.
It follows that $\m\Ext^1_R(M,R)=0$.
Thus $\m\Ext^1_R(\cm(R),R)=0$.
\end{proof}

We also make a comparison of $\tr_R(\omega_R)$ and $\tr_{R'}(\omega_{R'})$, where $R'=R/(\underline{x})$ for some regular sequence $\underline{x}$ in more general cases.
We prepare the following lemmas.

\begin{lem} \label{l23}
Let $R$ be a ring, and $M,N$ be $R$-modules.
Assume $x$ is a regular element on $M$ and $N$ such that $x\in \ann \Ext^t_R(M,N)\cap \ann\Ext^{t+1}_R(M,N)$.
Then there exists an isomorphism
$$
\Ext^t_{R/(x^u)}(M/x^uM,N/x^uN)\cong \Ext^t_R(M,N)\oplus \Ext^{t+1}_R(M,N)
$$
for any $u\ge 2$.
\end{lem}

\begin{proof}
Consider a commutative diagram
\begin{gather}\label{diag1}
\xymatrix{
0 \ar[r] & M \ar[r]^-{x^u} \ar[d]^-x & M \ar[r] \ar[d]^-{\id} &M/x^uM \ar[r] \ar[d] &  0\\
0 \ar[r] & M \ar[r]^-{x^{u-1}} & M \ar[r] & M/x^{u-1}M \ar[r] & 0 
}
\end{gather}
with exact rows.
Taking the long exact sequence induced by the functor $\Hom_R(-,N)$, we get a commutative diagram
\[
\xymatrix{
\cdots \ar[r]^-{x^{u-1}} & \Ext^t_R(M,N) \ar[r] \ar[d]^-x & \Ext^{t+1}_R(M/x^{u-1}M,N) \ar[r]^-\alpha \ar[d]^-\gamma & \Ext^{t+1}_R(M,N) \ar[r]^-{x^{u-1}} \ar[d]^{\id} & \cdots\\
\cdots \ar[r]^-{x^{u}} & \Ext^t_R(M,N) \ar[r] & \Ext^{t+1}_R(M/x^{u}M,N) \ar[r]^-\beta & \Ext^{t+1}_R(M,N) \ar[r]^-{x^{u}} & \cdots
}
\]
with exact rows.
Since $x^{u-1}$ and $x^{u}$ annihilates $\Ext^t_R(M,N)$ and $\Ext^{t+1}_R(M,N)$,
the $3$-terms of both rows become short exact sequences.
The homomorphism on left-terms is a multiplication by $x$, and hence it is zero.
Therefore $\gamma$ factors through $\alpha$, that is, there exists a homomorphism $\delta\colon \Ext^{t+1}_R(M,N) \to \Ext^{t+1}_R(M/x^uM,N)$ with $\delta\alpha=\gamma$.
It yields equalities $\beta\delta\alpha=\beta\gamma=\id\alpha$, and the surjectivity of $\alpha$ implies that $\beta\delta=\id$.
This shows a direct sum decomposition
$$
\Ext^{t+1}_{R}(M/x^uM,N)\cong \Ext^t_R(M,N)\oplus \Ext^{t+1}_R(M,N)
$$
Since $x^u$ is regular on $N$, the left-hand side is isomorphic to $$\Ext^t_{R/(x^{u})}(M/x^uM,N/x^uN)$$
 by \cite[Lemma 3.1.16]{BH}.
\end{proof}

The following lemma is essentially same as the proposition \cite[Proposition 6.15]{Y} given by Yoshino in the context of efficient system of parameters (see also \cite[Proposition 15.8]{LW}).
For our convenience, we give a proof.

\begin{lem} \label{l24}
Let $R$ be a ring, and $M,N$ be $R$-modules.
Assume $\underline{x}=x_1,\dots,x_t$ is a regular sequence on $M$ and $N$ such that $(\underline{x})\subseteq \ann \Ext^{i}_R(M,N)$ for all $i=1,\dots,t$.
Then for any homomorphism 
$\tilde{f}\colon M/\underline{x^{u+1}}M\to N/\underline{x^{u+1}}N$,
 where $u\ge 1$ is an integer, there exists a homomorphism $f\colon M \to N$ such that $f\otimes_RR/(\underline{x^u})$ is equal to $\tilde{f}\otimes_R R/(\underline{x^u})$.
\end{lem}

\begin{proof}
Fix a homomorphism $\tilde{f}\colon M/\underline{x^{u+1}}M\to N/\underline{x^{u+1}}N$.
First consider the case $t=1$, in order to start an induction argument.
Consider a commutative diagram same as \eqref{diag1}.
Then applying $\Hom_R(M,-)$ to obtain a commutative diagram
\[
\xymatrix{
\Hom_R(M,N) \ar[r] \ar[d]^-{\id} & \Hom_R(M,N/x^{u+1}N) \ar[r] \ar[d]^-\phi & \Ext^1_R(M,N) \ar[d]^-x\\
\Hom_R(M,N) \ar[r] & \Hom_R(M,N/x^{u}N) \ar[r] & \Ext^1_R(M,N)
}
\]
with exact rows.
Then one has $\Hom_R(M,N/x^{u+1}N)\cong \Hom_R(M/x^{u+1}M,N/x^{u+1}N)$, and so $\tilde{f}$ can be regarded as an element of $\Hom_R(M,N/x^{u+1}N)$.
By the assumption, the vertical map $x$ on the right of the diagram is zero.
Thus $\phi(\tilde{f})$ has a lift $f\in\Hom_R(M,N)$.
This means that $\tilde{f}\otimes_R R/(x^u)=f\otimes_R R/(x^u)$.

Now assume that the assertion holds for $1,\dots, t-1$ and we shall proceed an induction.
By Lemma \ref{l23}, $(\underline{x'}):=(x_2,\dots,x_t)$ annihilates $\Ext^i_{R/(x_1^{u+1})}(M/x_1^{u+1}M,N/x_1^{u+1}N)$ for all $i=1,\dots, t-1$.
Remark that the sequence $x_2,\dots,x_t$ is regular on $M/x_1^{u+1}M$ and $N/x_1^{u+1}N$.
Thus applying the induction hypothesis to the ring $R':=R/(x_1^{u+1})$, we obtain a homomorphism $f'\colon M/x_1^{u+1}M \to N/x_1^{u+1}N$ such that $f'\otimes_{R'} R'/(\underline{x'^u})$ is equal to $\tilde{f}\otimes_{R'} R'/(\underline{x'^u})$.
By the first step of the induction, we have a lift $f\colon M \to N$ of $f'\otimes_R R/(x_1^u)$.
Then $f$ satisfies the desired equality.
\end{proof}

We aim to apply the above lemma for a pair of modules $\omega, R$.

\begin{thm} \label{th26}
Let $R$ be a Cohen--Macaulay local ring with a canonical module $\omega$.
Let $\underline{x}=x_1,\dots,x_t$ be an $R$-regular sequence, $u\ge 1$ be an integer, $R'=R/(\underline{x^{u+1}})$ and $\pi\colon R\to R'$ be the natural surjection.
Assume that $\underline{x}$ is contained in $\ann \Ext_R^{i}(\omega,R)$ for all $i=1,\dots, t$.
\begin{enumerate}[\rm(1)]
\item $\tr_R(\omega_R)+(\underline{x^u})=\pi^{-1}(\tr_{R'}(\omega_{R'}))+(\underline{x^u})$.
\item If $\underline{x^u}$ is contained in $\tr_R(\omega)$, then $\res(R)=\res(R')$.
\end{enumerate}
\end{thm}

\begin{proof}
(1) It is enough to show that $\tr_R(\omega_R)R''=\tr_{R'}(\omega_{R'})R''$, where $R''=R/(\underline{x^u})$.
By Lemma \ref{l22}, $\tr_R(\omega_R)R'$ is contained in $\tr_{R'}(\omega_{R'})$.
Thus the inclusion $\tr_R(\omega_R)R''\subseteq \tr_{R'}(\omega_{R'})R''$ is obvious.
Take homomorphisms $\tilde{f_1},\dots,\tilde{f_s}\colon \omega_{R'} \to R'$ such that $\tr_{R'}(\omega_{R'})=\im \tilde{f_1}+\cdots+\im \tilde{f_s}$.
Note that $\omega_{R'}$ is isomorphic to $\omega_R\otimes_R R/(\underline{x^{u+1}})$.
Thanks to Lemma \ref{l24}, we have homomorphisms $f_1,\dots,f_s\colon \omega_R \to R$ such that
$f_i\otimes_RR''=\tilde{f_i}\otimes_RR''$ for all $i=1,\dots,s$.
Thus we get the following inclusions
$$
\tr_R(\omega_R)R'' \supseteq \sum_{i=1}^s \im (f_i\otimes_RR'') =\sum_{i=1}^s \im (\tilde{f_i}\otimes_RR'')=\tr_{R'}(\omega_{R'})R''.$$

(2) Assume $\underline{x^u}$ is contained in $\tr_R(\omega)$.
Then by Lemma \ref{l22}, $\tr_R(\omega_R)R'$ contains $(\underline{x^u})R'$.
Therefore by (1), we have $\tr_R(\omega_R)=\pi^{-1}(\tr_{R'}(\omega_{R'}))$.
It follows that $\res(R)=\ell(R'/\tr_R(\omega_R)R')=\ell(R'/\tr_{R'}(\omega_{R'}))=\res(R')$.
\end{proof}

\section{Annihilators of $\Ext^i_R(\omega,R)$}

In this section we aim to understand the relationship between $\Ext^i_R(\omega,R)$ and $\tr\omega$.
To state our main result in this section, we need the following lemma.

\begin{lem} \label{l32}
Let $R$ be a Cohen--Macaulay local ring with a canonical module $\omega$.
Assume that $R$ has Cohen--Macaulay type $2$ and is generically Gorenstein.
Then for any $i>0$, $\Ext^i_R(\omega,R)\cong \Ext^{i+1}_R(\Tr\omega,R)$ .
\end{lem}

\begin{proof}
A homomorphism $\sigma \colon\omega \to \omega^{**}$ given by $z\mapsto[f\mapsto f(z)]$ fits into an exact sequence
$$0\to \Ext^1_R(\Tr\omega,R) \to \omega \xrightarrow[]{\sigma} \omega^{**} \to \Ext^2_R(\Tr\omega,R) \to 0,$$
see \cite[Proposition 12.5]{LW}.
Since $R$ is generically Gorenstein and has type $2$, $\omega$ is isomorphic to a $2$-generated ideal $I=(a,b)$ of $R$ containing a regular element.
Then we have a short exact sequence 
$$0 \to \omega^* \xrightarrow[]{\alpha} R^{\oplus 2} \xrightarrow[]{(a,b)} \omega \to 0,$$
where $\alpha$ is given by an assignment $f\mapsto (f(b),-f(a))^T$; see \cite[Lemma 3]{HH} for instance.
Taking an $R$-dual of this sequence, we obtain a diagram
$$
\xymatrix{
0 \ar[r] & \omega^* \ar[r]^\alpha & R^{\oplus 2} \ar[r]^{(a,b)} \ar[d]^{\beta}& \omega \ar[r] \ar[d]^{\sigma} & 0 &\\
0 \ar[r] & \omega^* \ar[r] & R^{\oplus 2} \ar[r]^{\alpha^*} & \omega^{**} \ar[r] & \Ext^1_R(\omega,R) \ar[r] & 0
}
$$
, where $\beta$ is an automorphism $\left(\begin{smallmatrix}0&-1\\1&0\end{smallmatrix}\right)$, and isomorphisms $\Ext^i_R(\omega^*,R)\cong \Ext^{i+1}_R(\omega,R)$ for all $i>0$.
Since $\omega^*\cong \syz^2\Tr\omega$, we get isomorphisms $\Ext^{i+2}_R(\Tr \omega,R)\cong \Ext^{i+1}_R(\omega,R)$ for all $i>0$.
The commutativity of the square can be checked by diagram chases.
Thus by the snake lemma, we obtain isomorphisms $\Ext^2(\Tr\omega,R)\cong\cok \sigma\cong \Ext^1_R(\omega,R)$.
\end{proof}

In \cite{ABS,HH}, an analog of a conjecture of Tachikawa is discussed.
The statement of the conjecture is as follows.

\begin{conj}[Tachikawa]
Let $R$ be a Cohen--Macaulay local ring with a canonical module $\omega$.
If $\Ext^i_R(\omega,R)=0$ for all $i>0$, then $R$ is Gorenstein.
\end{conj}

Inspired by this conjecture, we put the following question on nearly Gorenstein rings.

\begin{ques}
Let $R$ be a Cohen--Macaulay local ring with a canonical module $\omega$.
Assume $\m\Ext^i_R(\omega,R)=0$ for all $i>0$.
Then is it true that $R$ is nearly Gorenstein?
\end{ques}

The following theorem gives a partial answer for the above question in the case of generically Gorenstein and Cohen--Macaulay type at most $2$.

\begin{thm}
Let $R$ be a $d$-dimensional Cohen--Macaulay local ring of Cohen--Macaulay type at most $2$ with a canonical module $\omega$.
Then $R$ is nearly Gorenstein if and only if $R$ is generically Gorenstein and $\m\Ext^i_R(\omega,R)=0$ for $i=1,\dots,d$.
\end{thm}

\begin{proof}
The ``only if" part follows by Lemma \ref{c14}.
We now aim to show the ``if" part.
Suppose $\m\Ext^i_R(\omega,R)=0$ for $i=1,\dots,d$.
Take a non-maximal prime ideal $\p$.
Then $\Ext^i_{R_\p}(\omega_{\p},R_\p)=0$ for all $i=1,\dots,d$.
By our assumption, $R_\p$ is generically Gorenstein.
Thus it follows from Lemma \ref{l32} that $\Ext^{i+1}_{R_\p}(\Tr \omega_{\p},R_\p)=0$ for all all $i=1,\dots,d$.
By Lemma \ref{l214}, it yields that $R_\p$ is Gorenstein.
Consequently, $R$ is Gorenstein on the punctured spectrum.
Then by Lemma \ref{l32} and Proposition \ref{1}, we obtain equalities
$$\tr \omega=\ann \Ext^{d+1}_R(\Tr \omega,R)=\ann \Ext^d_R(\omega,R)\supseteq\m.$$
It means that $R$ is nearly Gorenstein.
\end{proof}

\section{Weakly almost Gorenstein rings}

Let $R$ be a Cohen--Macaulay local ring with a canonical module $\omega$.
In the rest of this section, we denote by $r$ the Cohen--Macaulay type of $R$.
Recall that there is another notion of almost Gorenstein rings \cite{GTT}.

\begin{dfn}[Goto--Takahashi--Taniguchi]
$R$ is almost Gorenstein if there exists a short exact sequence
$$0\to R\to \omega \to C \to 0$$
such that $\mu(C)=\mathrm{e}(C)$.
Here $\mu(C)$ is the number of elements in a minimal system of generators of $C$, and $\mathrm{e}(C)$ is the multiplicity of $C$ with respect to $\m$.
We remark that if $\dim R=1$, then $C$ should be isomorphic to $k^{\oplus r-1}$.
\end{dfn}

Remark that almost Gorenstein rings are introduced by Barucci and Fr\"oberg \cite{BF} in the case where the local ring is analytically unramified and has dimension one.
Goto, Matsuoka and Phuong \cite{GMP} extended the notion to arbitrary local rings of dimension one.
The definition above coincides with Goto-Matsuoka-Phuong's definition when the local ring has dimension one and infinite residue field (\cite[Proposition 3.4]{GTT}).
We relax the above definition for small dimensions and introduce the notion of weakly almost Gorenstein rings.

\begin{dfn} \label{d52}
Let $(R,\m,k)$ be a Cohen--Macaulay local ring of Krull dimension at most $1$ with a canonical module $\omega$.
We say that $R$ is \textit{weakly almost Gorenstein} if there exists an exact sequence
$$R\to \omega \to k^{\oplus n} \to 0$$
for some $n\ge 0$.
\end{dfn}

One-dimensional almost Gorenstein rings are weakly almost Gorenstein.

\begin{prop} \label{p54}
Let $R$ be a Cohen--Macaulay local ring with a canonical module $\omega$ and Krull dimension $1$.
Then $R$ is almost Gorenstein  if and only if $R$ is weakly almost Gorenstein.
\end{prop}

\begin{proof}
The ``only if" part is clear.
Suppose $\dim R>0$ and there is an exact sequence $R\xrightarrow[]{\varphi} \omega \to k^{\oplus n} \to 0$.
Then $\varphi$ is injective by \cite[Lemma 3.1]{GTT}.
Therefore $R$ is almost Gorenstein.
\end{proof}

We can also see another simple example as follow.

\begin{ex}
Let $(R,\m)$ be an artinian local ring with $\m^2=0$ (for example, $R=k[x_1,\dots,x_n]/(x_1,\dots,x_n)^2$, where $k$ is a field).
Then $R$ is weakly almost Gorenstein.
Indeed, let $\omega$ be a canonical module of $R$.
Then we have an equality $\m^2\omega=0$, and hence $\soc \omega=\m\omega$.
Since $\soc \omega$ has length one, we can define a surjective homomorphism $s\colon R\to \soc \omega$.
Composing $s$ with an inclusion $\soc \omega \to \omega$, we get a homomorphism $\varphi\colon R\to \omega$, whose image contains $\soc \omega=\m\omega$.
It follows that $\m \cok \varphi=0$, i.e. $\cok\varphi\cong k^{\oplus n}$ for some $n$.
\end{ex}

Now we give a proof of Theorem \ref{thC}.

\begin{proof}[Proof of Theorem \ref{thC}]
(1) $\iff$ (2): This equivalence is shown in Proposition \ref{p54}.

(2) $\implies$ (3): Fix a regular element $x$.
Tensoring $R/(x)$ with the exact sequence $R\to\omega \to k^{\oplus n} \to 0$ over $R$, we get an exact sequence $R/(x) \to \omega_{R/(x)} \to k^{\oplus n} \to 0$.
Thus $R/(x)$ is weakly almost Gorenstein.
The implication (3)$\implies$(4) is clear.
(4) $\implies$ (2): Assume $x\in\m^2$ is a regular element, and $R/(x)$ is weakly almost Gorenstein.
Then there exists an element $a\in \omega_{R/(x)}$ with an isomorphism $\omega_{R/(x)}/(a)\cong k^{\oplus n}$.
Pick a lift $b\in\omega_R$ of $a$ along the natural surjection $\omega \to \omega/x\omega=\omega_{R/(x)}$.
It follows that $$\omega/((b)+x\omega)\cong(\omega/(b))\otimes_R (R/(x))\cong \omega_{R/(x)}/(a)\cong k^{\oplus n}.$$
Hence we obtain an inequality $(b)+x\omega \supseteq \m\omega$.
Since $x\in\m^2$, Nakayama's lemma implies that $(b)\supseteq \m\omega$.
Thus $R$ is weakly almost Gorenstein.
\end{proof}

The following implication is found in \cite[Proposition 6.1]{HHS}.

\begin{thm}[Herzog--Hibi--Stamate \cite{HHS}]
Let $R$ be a Cohen--Macaulay local ring with a canonical module $\omega$.
Assume $R$ is almost Gorenstein and has Krull dimension one.
Then $R$ is nearly Gorenstein.
\end{thm}

There exists a weakly almost Gorenstein ring which is not nearly Gorenstein (see Example \ref{ex5}).
On the other hand, for local rings with type two, an analogue of the result of Herzog--Hibi--Stamate holds as follows.

\begin{thm} \label{th5}
Let $R$ be a Cohen--Macaulay local ring of Cohen--Macaulay type two with a canonical module $\omega$.
Assume $R$ is weakly almost Gorenstein.
Then $R$ is nearly Gorenstein.
\end{thm}

To prove this theorem, we need some lemmas.

\begin{lem} \label{l57}
Let $(R,\m)$ be an artinian local ring.
Then $R/\soc R$ is Gorenstein if and only if $R$ is isomorphic to $S/I$, where $(S,\n)$ is a regular local ring and $I$ is an $\n$-primary ideal $(x_1^p)+\n(x_2,\dots,x_n)$ for some $p\ge 2$ and minimal generators $x_1,\dots,x_n$ of $\n$.
In this case, $R$ has Cohen--Macaulay type $n$.
\end{lem}

\begin{proof}
Suppose that $R/\soc R$ is Gorenstein.
Take a Cohen presentation $R=S/I$, i.e. $(S,\n)$ is a regular local ring and $I\subseteq \n^2$ is an ideal.
Then $R/\soc R$ is isomorphic to $S/(I:_S\n)$.
The ideal $I:_S\n$ is Burch (See \cite[Corollary 2.4]{DKT}), and our assumption says that $S/(I:_S\n)$ is Gorenstein.
It follows from \cite[Theorem 4.4]{DKT} that there exists minimal generators $x_1,\dots,x_n$ of $\n$ and some $q\ge 1$ such that the equality $I:_S\n=(x_1^q,x_2,\dots,x_n)$ holds.
Then by inclusions $\n(I:_S\n)\subseteq I\subseteq \n^2\cap (I:_S\n)$, we can see that $I=(x_1^{q+1})+\n(x_2,\dots,x_n)$.

Conversely, assume $R=S/I$, where $(S,\n)$ is a regular local ring and $I$ is $(x_1^p)+\n(x_2,\dots,x_n)$ for some $p\ge 2$ and minimal generators $x_1,\dots,x_n$ of $\n$.
Then an equality $I:_S \n=(x_1^{p-1})+(x_2,\dots,x_n)$ holds.
It follows that $R/\soc R$ is isomorphic to an artinian principal ideal ring $S/(I:_S\n)\cong (S/(x_2,\dots,x_n))/(x_1^{p-1})$, which is Gorenstein.

Finaly, we note that the Cohen--Macaulay type of $R$ is equal to $\ell((I:_S\n)/I)=n$.
\end{proof}

\begin{lem} \label{l58}
Let $(R,\m)$ be a weakly almost Gorenstein ring with a canonical module $\omega$.
Then there exists an exact sequence $R\xrightarrow[]{\varphi}\omega \to k^{\oplus r-1} \to 0$, where $r$ is a Cohen--Macaulay type of $R$.
In this case, the kernel of $\varphi$ is generated by $m$-elements in $\soc R$, where $m=\max\{\dim_k \soc R-1,0\}$.
\end{lem}

\begin{proof}
We may assume that $R$ is non-Gorenstein.
If $\dim R>0$, then the assertion follows by Theorem \ref{thC} and \cite[Corollary 3.10]{GTT}.
Thus we may suppose that $R$ is artinian.
By definition, there exists an exact sequence $R\xrightarrow[]{\varphi}\omega \to k^{\oplus n} \to 0$ for some $n$.

Then the image $\im\varphi$ contains $\m\omega$.
Taking a canonical dual $(-)^\dag$ to an exact sequence $0\to \soc R \to R \to R/\soc R\to 0$, we see that $\m\omega=(R/\soc R)^\dag$.
In particular, $\ann(\m\omega)=\soc R$.
It follows that $\ker\varphi=\ann(\im\varphi)\subseteq\ann(\m\omega)=\soc R$.
Remark that $\ell_R(\ker \varphi)=\ell(R)-\ell(\omega)+\ell(k^{\oplus n})=\ell(k^{\oplus n})=n$.
Hence once we get $n=r-1$, the latter half assertion also follows.
To show that $n=r-1$, it is enough to see that $\im\varphi\not=\m\omega$.
Thus we only have to deal with the case that $\im\varphi=\m\omega$.
In this case, $\ker\varphi=\ann\im\varphi=\soc R$.
Thus $R/\soc R=R/\ker\varphi$ is isomorphic to a submodule of $\omega$.
It means that $R/\soc R$ has Cohen--Macaulay type one.
It follows from Lemma \ref{l57} that $R$ is isomorphic to $S/I$, where $(S,\n)$ is a regular local ring and $I=(x_1^p)+\n(x_2,\dots,x_r)$ for some $p\ge 2$ and some minimal generators $x_1,\dots,x_r$ of $\n$.
We may set $R=S/I$.
Let $J=(x_1^p)+(x_2,\dots,x_s)$ be an ideal of $S$.
Then $J$ contains $I$, and contained in $I:_S\n$.
Note that $S/J$ is an artinian principal ideal ring, and hence a Gorenstein ring.
Therefore as an $R$-module, $R/JR$ has Cohen--Macaulay type one.
It then follows that there exists an injective homomorphism $\psi\colon R/JR\to \omega_R$.
The kernel of $\psi$ is $JR$, which is a submodule of $\soc R=(I:_S\n)/I$ of length $r-1$.
Thus we have an exact sequence $0\to k^{\oplus r-1}\to R \to \omega$.
Taking the canonical dual of this sequence, we get the desired exact sequence $R\to \omega \to k^{\oplus r-1} \to 0$.
\end{proof}

The following lemma is essentially same as \cite[Theorem 7.8]{GTT}.
We give a proof within our context.

\begin{lem} \label{l48}
Let $(R,\m)$ be a non-Gorenstein Cohen--Macaulay local ring with a canonical module $\omega$.
Then $R$ is weakly almost Gorenstein if and only if $\omega$ has a free presentation
$R^{\oplus q} \xrightarrow[]{A} R^{\oplus r} \to \omega \to 0$ with an $q\times r$ matrix 
$$A=\left(\begin{array}{ccccccccccccc}
y_{21} \dots y_{2n} &
y_{31} \dots y_{3n} &
\dots &
y_{r1} \dots y_{rn} &
z_1 \dots z_m\\
x_1 \dots x_n &
0 &
0 &
0 &
0\\
0&
x_1 \dots x_n &
0 &
0 &
0\\
\vdots &
\vdots &
\ddots &
\vdots &
\vdots \\
0 &
0 &
0 &
x_1 \dots x_n &
0 
\end{array}
\right),
$$
where $n=\mathrm{edim} R$, $m=\max\{\dim_k \soc R-1,0\}$, $q=(r-1)n+m$, $z_1,\dots,z_m$ are elements in $\soc R$, and $x_1,\dots,x_n$ is a system of minimal generators of $\m$.
\end{lem}

\begin{proof}
Suppose that $R$ is weakly almost Gorenstein.
Then by Lemma \ref{l58}, there is an exact sequence $R\xrightarrow[]{\varphi}\omega \to k^{\oplus r-1} \to 0$, and the kernel of $\varphi$ is generated by $m$-elements $z_1,\dots,z_m$ of $\soc R$.
Hence $R^{\oplus m}\xrightarrow[]{(z_1,\dots,z_m)} R \to \im\varphi \to 0$ is a free presentation.
Note that $k$ has a free presentation $R^{\oplus n}\xrightarrow[]{(x_1,\dots,x_n)} R\to k \to 0$.
Applying the horseshoe lemma to the short exact sequence $0\to \im\varphi \to \omega \to k^{\oplus r-1} \to 0$, we obtain a free presentation $R^{\oplus q}\xrightarrow[]{A}R^{\oplus r}\to \omega \to 0$ with the desired matrix $A$.

Conversely, suppose that $\omega$ has a free presentation $R^{\oplus q}\xrightarrow[]{A} R^{\oplus r}\xrightarrow[]{\varepsilon}\omega \to 0$, where $A$ is a matrix as in the statement.
Let $e_1,\dots,e_r$ be a standard $R$-basis of the free module $R^{\oplus r}$.
Then for any $i\in\{1,\dots,n\}$ and $j\in\{1,\dots,r-1\}$, $x_ie_j-y_{j+1i}e_r$ is in the image of $A$.
This means that $x_i\varepsilon(e_j)=y_{j+1i}\varepsilon(e_r)$.
Since $\omega$ is generated by $\varepsilon(e_1),\dots,\varepsilon(e_r)$, it follows that $\m\omega\subseteq (\varepsilon(e_r))$.
Let $\varphi$ be a homomorphism $R\to \omega$ sending $1$ to $\varepsilon(e_r)$.
Then we see that $\im \varphi\supseteq \m\omega$.
This implies that $\cok \varphi$ is isomorphic to a sum of copies of $k$.
It shows that $R$ is weakly almost Gorenstein.
\end{proof}

\begin{proof}[Proof of Theorem \ref{th5}]
Take a free presentation $G \xrightarrow[]{A} F \xrightarrow[]{\epsilon} \omega \to 0$ of $\omega$, where $A$ is a matrix given in Lemma \ref{l48}.
Since $R$ has Cohen--Macaulay type $2$, $F$ has rank $2$.
Then we have a free resolution $\cdots \to H \xrightarrow[]{B} F^* \xrightarrow[]{A^T} G^* \to \Tr \omega \to 0$ of $\Tr \omega$.
By \cite[Remark 3.3]{Vas}, $\tr(\omega)$ is equal to $I_1(B)$.

Let $e_1,e_2$ be a standerd basis and set $w_1=\epsilon(e_1)$ and $w_2=\epsilon(e_2)$.
For each $i$, one has an equality $y_{2i}w_1=x_iw_2$
Thus for each pair $i,j$ with $i\not=j$, the element $y_{2i}x_j-y_{2j}x_i$ annihilates $w_1$ and $w_2$. 
Since $\omega$ is faithful, one has $y_{2i}x_j-y_{2j}x_i$.
Consider vectors $v_1=(x_1,-y_{21})^T, v_2=(x_2,-y_{22})^T,\dots,v_n=(x_n,-y_{2n})^T$ in $F^*$.
It follows that $A^T(v_i)=0$, which means that the columns of $B$ generate $v_1,\dots,v_n$.
In particular, one has $I_1(B)\supseteq I_1(v_1)+\cdots I_1(v_n)\supseteq (x_1,\dots,x_n)=\m$.
We conclude that $R$ is nearly Gorenstein.
\end{proof}

Recall that an artinian local ring $R$ is called \textit{Teter} or a \textit{Teter's ring}
if it is isomorphic to $S/\soc S$ for some artinian Gorenstein local ring $S$.
It is shown in \cite[Corollary 2.2]{HV} that Teter rings are nearly Gorenstein.
We give the following characterization of weakly almost Gorenstein rings in terms of Teter rings.

\begin{thm} \label{th511}
Let $(R,\m,k)$ be an artinian local ring.
Consider the following conditions.
\begin{enumerate}[\rm(1)]
\item $R$ is weakly almost Gorenstein.
\item There exists an ideal $I\subseteq \soc R$ such that $R/I$ is Gorenstein.
\item $R/\soc R$ is Teter.
\item $R/\soc R$ has self-canonical dual maximal ideal (i.e. $\Hom_R(\m/\soc R,\omega)\cong \m/\soc R$.)
\end{enumerate}
Then implications (1)$\Leftrightarrow$(2)$\Rightarrow$(3)$\Rightarrow$(4) hold.
If $k$ is infinite, then an implication (4)$\Rightarrow$(1) also holds.
\end{thm}

In the proof of this theorem, we use the following lemma.

\begin{lem} \label{l51}
Let $(R,\m)$ be a local ring.
Let $\m=\m_1\oplus\m_2\oplus\cdots\oplus \m_n$ be a direct sum decomposition.
Then for any $R$-homomorphism $\alpha\colon\m_i\to\m_j$, its image $\im \alpha$ is contained in $\soc(\m_j)$ for any $i,j$ with $i\not=j$.
\end{lem}

\begin{proof}
For any $i\not=j$, we have $\m_i\m_j\subseteq \m_i\cap \m_j=0$.
Thus $\ann(\m_i)\supseteq \m_1+\cdots+\check{\m_i}+\cdots+\m_n$.
It follows that $\ann \im \alpha\supseteq \ann(\m_i)+\ann(\m_j)\supseteq \m_1+\cdots+\check{\m_i}+\cdots+\m_n+\m_i=\m$.
This shows that $\im\alpha\subseteq \soc(\m_j)$.
\end{proof}

We are now ready to prove Theorem \ref{th511}.

\begin{proof}[Proof of Theorem \ref{th511}]
(1)$\implies$(2): Take an exact sequence $R\xrightarrow[]{\varphi} \omega \to k^{\oplus n} \to 0$.
Then $\im\varphi$ contains $\m\omega$.
Thus $\ker\varphi=\ann\im\varphi$ is contained in $\soc R$.
Set $I:=\ker\varphi$.
Then $R/I\cong \im\varphi\subseteq \omega$, which implies that $\soc (R/I)$ has dimension one over $k$.
This means that $R/I$ is Gorenstein.
(2)$\implies$(1): Consider a short exact sequence $0\to I \to R \to R/I \to 0$.
Taking the canonical dual of this, we get an exact sequence $0\to \Hom_R(R/I,\omega)\to \omega \to \Hom_R(I,\omega) \to 0$.
Since $I$ is in $\soc R$, $\Hom_R(I,\omega)$ is isomorphic to $k^{\oplus n}$ for some $n$.
The Gorensteiness of $R/I$ implies that $\Hom_R(R/I,\omega)\cong R/I$.
Thus we obtain an exact sequence $0\to I \to R\to \omega \to k^{\oplus n}\to 0$, which shows that $R$ is weakly almost Gorenstein.
(1)$\implies$(3): We may assume $I\not=\soc R$.
Then one can see that $R/\soc R\cong (R/I)/(\soc(R/I))$.
Therefore $R/\soc R$ is Teter.
(3)$\implies$(4) See \cite{An} for instance.

Now assume $k$ is infinite.
(4)$\implies$(2): Set $(-)^\dag:=\Hom_R(-,\omega)$.
Consider an $R$-free presentation $R^{\oplus m}\to R^{\oplus n} \xrightarrow[]{(f_1,\dots,f_n)} \m^\dag \to 0$ of $\m^\dag$.
Here $f_i$ is a non-zero homomorphism $R\to \m^\dag$.
Note that $\bigcap_i \ker f_i=\ann(\m^\dag)=\ann(\m)=\soc R$.
We set $g_i:=f_i(1)$ to be a homomorphism from $\m$ to $\omega$.
The kernel $I$ of a linear combination $g:=a_1g_1+\cdots+a_ng_n\colon \m \to \omega$ with $a_i\in R$ gives a Gorenstein quotient $R/I$.
Remark that $I$ is contained in $\ker f$, where $f:=a_1f_1+\cdots+a_nf_n$.
Thus we want to find elements $a_1,\dots,a_n$ such that $\ker f$ is contained in $\soc R$.
Let $h_i$ be homomorphisms $\m\to \m\cdot\m^\dag$ which are restrictions of $f_i$, and $h$ be a sum $a_1h_1+\cdots+a_nh_n$.
Then $\m^{\oplus n}\xrightarrow[]{(h_1,\dots,h_n)} \m\cdot\m^\dag\to 0$ is exact.
We also have $\ker f=\ker h$.
Thus we aim to find elements $a_1,\dots,a_n$ such that $\ker h$ is contained in $\soc R$.
Put $X=\m/\soc R$.
Since $\soc R$ annihilates $\m^\dag$, $h_i$ induces a homomorphism $\bar{h}_i\colon X \to \m\cdot \m^\dag$ for any $i$.
There is a short exact sequence $0\to \soc R \to \m \to X \to 0$.
Applying $(-)^\dag$, we get a short exact sequence $0 \to X^\dag \to \m^\dag \to (\soc R)^\dag \to 0$.
It follows that $X^\dag\cong \m\cdot\m^\dag$.
By our assumption, $X\cong X^\dag$.
Therefore, there is an isomorphism $\varphi\colon\m\cdot\m^\dag\to X$.
Now we have endomorphisms $\varphi\circ \bar{h}_i$ of $X$.
Take a decomposition $X=X_1\oplus X_2\oplus \cdots \oplus X_m$ with indecomposables $X_i$, and assume that $X_j\cong k$ if and only if $j=s+1,\dots,m$ (we admit $s>m$).
Let $\iota_j\colon X_j\to X$ be inclusions and $p_j\colon X\to X_j$ be projections.
Set $h_{ijl}:=p_l\circ \varphi \circ\bar{h}_i \circ\iota_j\colon X_j\to X_l$.

\begin{claim} \label{cl1}
If $j\not=l$ and $l\le s$, then the image $\im h_{ijl}$ is contained in $\m X_l$ for any $i$.
\end{claim}

\begin{proof}[Proof of Claim \ref{cl1}]
Note that $X=\m/\soc R$ is a maximal ideal of $R/\soc R$.
By Lemma \ref{l51}, $\im h_{ijl}$ is contained in $\soc(X_l)$.
Since $X_l$ is indecomposable and not isomorphic to $k$, $\soc(X_l)$ is contained in $\m X_l$.
\end{proof}

\begin{claim} \label{cl2}
For any $l=1,\dots,s$, there is an index $i$ such that $h_{ill}\colon X_l\to X_l$ is an automorphism.
\end{claim}

\begin{proof}
Fix an integer $l\in\{1,\dots,s\}$.
Since $(\bar{h}_1,\dots,\bar{h}_n)\colon X^{\oplus n}\to \m\cdot\m^\dag$ is surjective,
we obtain an equality $\sum_{i,j}\im h_{ijl}=X_l$.
By Claim \ref{cl1}, $\im h_{ijl}$ is contained in $\m X_l$ for any $i,j$ with $j\not=l$.
It follows by Nakayama's lemma that $\sum_i \im h_{ill}=X_l$.
Suppose that $h_{ill}$ is not an automorphism for any $i$.
As $X_l$ is indecomposable, $\End_R(X_l)$ is local.
Thus $h_{ill}$ is contained in $J:=\mathrm{Jac}(\End_R(X_l))$, the Jacobson radical of $\End_R(X_l)$.
This yields that $\im h_{ill}\subseteq JX_l$.
It follows that $\sum_i \im h_{ill}\subseteq JX_l$.
By Nakayama's lemma, it implies that $\sum_i\im h_{jll} \not= X_l$, a contradiction.
Therefore there is an index $i$ with an automorphism $h_{ill}$.
\end{proof}

In particular, Claim \ref{cl2} shows that the Zariski open set 
$$U_j:=\{(b_1,\dots,b_n)\in k^{\oplus n}\mid \det(b_1(h_{1jj}\otimes_R k)+\cdots+b_n(h_{njj}\otimes_R k))\not=0\}$$
 is non-empty for any $j=1,\dots,s$.
Since $\bigcap_i \ker h_i=\bigcap_i \ker f_i=\soc R$, one has $\bigcap_i \ker \bar{h}_i=0$.
From this equality and an isomorphism $X_j\cong k$ for $j>s$, we see that for any $j>s$, there exists $i(j)$ such that $\bar{h}_{i(j)}\circ \iota_l\colon X_j\to \m\cdot\m^\dag$ is injective.
Now we set $U_j:=\{(b_1,\dots,b_n)\in k^{\oplus n}\mid b_{i(j)\not=0}\}$ for any $j>s$.
Then $U_j$ is non-empty Zariski open set for any $j=1,\dots,s,s+1,\dots m$.
Since $k$ is infinite, $U:=U_1\cap U_2\cap\cdots\cap U_m$ is also non-empty.
Take an element $(a_1,\dots,a_n)\in R^{\oplus n}$ such that its image in $(R/\m)^{\oplus n}=k^{\oplus n}$ belongs to $U$.
Set a homomorphism $\bar{h}=a_1\bar{h}_1+\cdots+a_n\bar{h}_n$ from  $X$ to $\m\cdot\m^\dag$.
Then by the definition of $U$, $(p_j\circ\varphi\circ\bar{h}\circ \iota_j)\otimes_R k$ is an automorphism on $X_j/\m X_j$ for any $j=1,\dots,s$.
This yields that $p_j\circ\varphi\circ\bar{h}\circ \iota_j$ is also an automorphism on $X_j$.
This particularly shows that $\bar{h}\circ \iota_j$ is injective.
On the other hand, $\bar{h}\circ \iota_j\colon X_j\to \m\cdot\m^\dag$ is injective for $j>s$.
Thus $\bar{h}$ is injective on any summand of $X$.
This says that $\bar{h}$ is injective.
Now we get an equality $\ker(a_1h_1+\cdots+a_nh_n)=\soc R$, which completes the proof.
\end{proof}

\begin{cor} \label{c5}
Let $(S,\n)$ be a local ring and $I\subseteq \n^2$ be an $\n$-primary ideal of $S$.
Assume that $S/I$ is Gorenstein.
Then for any ideal $J$ of $S$ such that $\n I\subseteq J\subseteq I$, $S/J$ is weakly almost Gorenstein.
\end{cor}

\begin{proof}
We may assume that $\dim S\ge 2$.
Fix an ideal $J$ with inclusions $\n I\subseteq J\subseteq I$.
Set $R=S/J$.
Then since $I\subseteq \n^2$ and $S/I$ is Gorenstein, $I$ is not a Burch ideal \cite[Theorem 4.4]{DKT}.
It means that an equality $\n I:_S \n=I:_S\n$ holds.
Then we see from the inclusions $\n I:_S \n\subseteq J:_S\n\subseteq I:_S\n=\n I:_S \n$ that $I\subseteq I:_S\n=J:_S\n$.
This yields that $IR$ is an ideal of $R$ contained in $\soc R=(J:_S\n)/J$.
Note that $R/IR=S/I$ is Gorenstein by the assumption.
Thus Theorem \ref{th511} shows that $R$ is weakly almost Gorenstein.
\end{proof}

Using Corollary \ref{c5}, we can provide many examples of weakly almost Gorenstein rings.

\begin{ex} \label{ex5}
Let $(S,\n)$ be a power series ring $k[\![x,y]\!]$ of two variables over a field $k$.
Set $I=(x^2,y^2)$.
Then $S/I$ is Gorenstein.
Put $J=(x,y)^3$.
Then an equality $J=\n I$ holds, and hence $S/J$ is weakly almost Gorenstein.
We remark that $S/J$ is not nearly Gorenstein by \cite[Theorem 3.2]{An2}.
\end{ex}

\begin{ex} \label{e51}
Let $S$ be a power series ring $k[\![x,y,z]\!]$ of three variables over a field $k$.
Let $\n$ denote the maximal ideal $(x,y,z)S$ of $S$.
Fix integers $a>1,b>1,c>1$.
Let $I$ be a monomial ideal $(x^a,y^b,z^c)$.
Note that $S/I$ is Gorenstein.
Assume $J$ is an ideal of one of the following forms.
\begin{enumerate}[\rm(a)]
\item $(x^a,y^b)+\n(z^c)=(x^a,y^b,xz^c,yz^c,z^{c+1})$.
\item $(x^a)+\n(y^b,z^c)=(x^a,xy^b,y^{b+1},y^bz,xz^c,yz^c,z^{c+1})$.
\item $\n I=(x^{a+1},x^ay,x^az,xy^b,y^{b+1},y^bz,xz^c,yz^c,z^{c+1})$.
\end{enumerate}
Then $J$ satisfies $\n I\subseteq J\subseteq I$, and hence $S/J$ is weakly almost Gorenstein by Corollary \ref{c5}.
\end{ex}

\begin{ex} \label{e52}
Let $S$ be a power series ring $k[\![x,y,z]\!]$ of three variables over a field $k$.
Let $\n$ denote the maximal ideal $(x,y,z)S$ of $S$.
Fix an integer $m\ge 2$.
L. Christensen, O. Veliche and J. Weyman \cite[Proposition 3.3]{CVW} give an $\n$-primary Gorensten ideal $\mathfrak{g}_m\subseteq \n^2$ minimally generated by $2m+1$ elements.
Let $g$ be one of the minimal generators of $\mathfrak{g}_m$ listed in \cite[Proposition 3.3]{CVW}, let $\mathfrak{b}$ be the ideal generated by remaining generators of $\mathfrak{g}_m$, and put $\mathfrak{a}=\mathfrak{n}g+\mathfrak{b}$.
Then $\mathfrak{a}$ satisfies $\n\mathfrak{g}_m\subseteq \mathfrak{a}\subseteq \mathfrak{g}_m$, and hence $S/\mathfrak{a}$ is weakly almost Gorenstein by Corollary \ref{c5}.

Remark that in \cite[Proposition 3.5]{CVW}, the authors also determine the structure of Koszul homology algebra of $S/\mathfrak{a}$ along classification results of \cite{AKM,W}.

\end{ex}
\begin{ac}
The authors are grateful to the anonymous referees for reading the paper carefully and giving useful comments and helpful suggestions.
\end{ac}

\end{document}